\documentclass[11pt,a4paper]{article}
\usepackage[T1]{fontenc}
\usepackage[bitstream-charter]{mathdesign}
\usepackage[latin1]{inputenc}	
\usepackage{amsmath}									% Math
\usepackage{amsthm}
\usepackage{amsbsy}
\usepackage[math]{easyeqn}
\usepackage{esint}
\usepackage{graphicx}

\usepackage{xcolor}
\definecolor{bl}{rgb}{0.0,0.2,0.6} 

\usepackage[authoryear]{natbib}
\usepackage{comment}

\usepackage{sectsty}
\usepackage[compact]{titlesec} 
\allsectionsfont{\color{bl}\scshape\selectfont}

\numberwithin{equation}{section}
\numberwithin{figure}{section}
\newcounter{example}[section]
\numberwithin{example}{section}
\newcounter{remark}[section]
\numberwithin{remark}{section}
\newtheorem{theorem}{Theorem}[section]
\newtheorem{proposition}[theorem]{Proposition}
\newtheorem{lemma}[theorem]{Lemma}
\newtheorem{corollary}[theorem]{Corollary}

\newtheorem{exmp}[example]{Example}
\newtheorem{rmrk}[remark]{Remark}
\newenvironment{example}{\begin{exmp}\rm}{\end{exmp}}
\newenvironment{remark}{\begin{rmrk}\rm}{\end{rmrk}}

\makeatletter							% Begin definition
\def\printtitle{%						% Define command: \printtitle
    {\color{bl} \centering \huge \sc \textbf{\@title}\par}}		% Typesetting
\makeatother							% End definition

\title{Concentration inequalities for smooth random fields
\vspace{10pt}
\footnote{\footnotesize This research was partially supported by the Deutsche
      Forschungsgemeinschaft through the CRC 823 ``Statistics of nonlinear dynamical processes'', the CRC 649  ``Economic Risk'' and   by
Laboratory for Structural Methods of Data Analysis in Predictive Modeling, MIPT, RF government grant, ag. 11.G34.31.0073.}
}

\makeatletter							% Begin definition
\def\printauthor{%					% Define command: \printauthor
    {\centering \small \@author}}				% Typesetting
\makeatother							% End definition

\author{%
	Denis Belomestny and Vladimir Spokoiny \\
	Duisburg-Essen University and WIAS Berlin, Moscow Institute of Physics and Technology \\
	\vspace{20pt}
	}

\usepackage{fancyhdr}
\usepackage{lastpage}	
\usepackage[runin]{abstract}			% runin option for a run-in title
\abslabeldelim{\quad}						% 
\def\rdl{\epsilon}
\def\rd{\bb{\rdl}}
\def\rddelta{\delta}
\def\rdomega{\varrho}

\def\rdb{\rd}
\def\rdm{\underline{\rdb}}

\def\Id{I\!\!\!I}
\def\Ind{\operatorname{1}\hspace{-4.3pt}\operatorname{I}}

\def\muc{\mu_{c}}

\def\gmd{\gm_{0}}

\def\rhor{\omega}

\def\DP{D}
\def\DPc{\DP_{0}}
\def\DPb{\DP_{\rdb}}
\def\DPm{\DP_{\rdm}}

\def\gmi{\mathtt{b}}

\def\xivb{\xiv_{\rdb}}
\def\xivm{\xiv_{\rdm}}

\def\pen{\mathfrak{t}}

\def\ex{\mathrm{e}}
\def\entrl{\mathbb{Q}}
\def\entrlb{\entrl}

\def\gm{\mathtt{g}}
\def\gmc{\gm_{c}}
\def\gmb{\gm}

\def\yy{\mathtt{y}}
\def\yyc{\yy_{c}}
\def\xx{\mathtt{x}}
\def\xxc{\xx_{c}}

\def\rups{\rr_{0}}

\def\CS{\cc{E}}

\def\nunu{\nu_{0}}

\def\dist{d}

\def\cdimb{\mathfrak{c}_{1}}

\def\rdomega{\varrho}

\def\err{\diamondsuit}

\def\errm{\err_{\rdm}}
\def\errb{\err_{\rdb}}

\def\pnnd{\mathfrak{u}}

\def\dimp{p}

\def\dimA{\mathtt{p}_{0}}

\def\BB{I\!\!B}
\def\vA{\mathtt{v}}

\def\thetav{\bb{\theta}}
\def\thetavs{\thetav^{*}}

\def\ups{\bb{\upsilon}}
\def\upss{\ups_{0}}
\def\upsc{\ups^{\prime}}
\def\upsd{\ups^{\circ}}

\def\Ups{\varUpsilon}
\def\Upsd{\Ups^{\circ}}
\def\Upss{\Ups_{\circ}}

\def\Thetas{\Theta_{0}}

\def\UP{\cc{U}}
\def\up{\mathfrak{u}}

\def\VP{V}
\def\VPc{\VP_{0}}

\def\VV{H}

\def\VVc{\VV_{0}}

\def\mes{\pi}

\def\lambdaB{{\lambda}^{*}}

\def\fis{\mathfrak{a}}

\def\Ldrift{M}

\def\B{\cc{B}}

\def\mubc{\mu}

\def\Mubc{\mathbb{M}}

\def\zzc{\zz_{c}}
\def\ww{w}
\def\wwc{\ww_{c}}

\def\Lmgf{\mathfrak{M}}
\def\Lmgfb{\Lmgf^{*}}

\def\rr{\mathtt{r}}
\def\rrb{\rr^{*}}

\def\zz{\mathfrak{z}}

\def\zzQ{\zz_{0}}

\def\dimA{\mathtt{p}}

\def\dimA{\mathtt{p}}

\renewcommand{\(}{$\,}
\renewcommand{\)}{\,$}

\def\nquad{\hspace{-1cm}}
\def\eqdef{\stackrel{\operatorname{def}}{=}}

\newcommand{\cc}[1]{\mathscr{#1}}
\newcommand{\bb}[1]{\boldsymbol{#1}}

\renewcommand{\tilde}[1]{\widetilde{#1}}

\renewcommand{\Gamma}{\varGamma}
\renewcommand{\Pi}{\varPi}
\renewcommand{\Sigma}{\varSigma}
\renewcommand{\Delta}{\varDelta}
\renewcommand{\Lambda}{\varLambda}
\renewcommand{\Psi}{\varPsi}
\renewcommand{\Phi}{\varPhi}
\renewcommand{\Theta}{\varTheta}
\renewcommand{\Omega}{\varOmega}
\renewcommand{\Xi}{\varXi}
\renewcommand{\Upsilon}{\varUpsilon}

\def\Var{\operatorname{Var}}

\def\argmax{\operatornamewithlimits{argmax}}

\def\tr{\operatorname{tr}}

\def\R{I\!\!R}
\def\E{I\!\!E}
\def\P{I\!\!P}

\def\kappa{\varkappa}

\def\T{\top}

\def\diam{\operatorname{diam}}

\def\gammav{\bb{\gamma}}

\def\xiv{\bb{\xi}}

\def\Za{\mathbb{Z}}
\def\Zab{\Za_{\rdb}}
\def\Zam{\Za_{\rdm}}

\begin{document}
%%% Top of the page: Author, Title and Abstact

\printtitle

\printauthor

\begin{abstract}
In this note we derive a sharp concentration inequality for the supremum of a smooth random field over a finite dimensional set.  It is shown that this supremum can be bounded with high probability by  the value of the field at some deterministic point plus an intrinsic  dimension of the optimisation problem. As an application we prove the exponential inequality for a function of the maximal eigenvalue of a random matrix is proved.
\end{abstract}

\section{Introduction}
Concentration inequalities is
a quite active field of research, which is driven by numerous applications, see \citet{L} and \citet{Lu} for an overview. 
Concentration inequalities have been used in many fields of both pure
and applied mathematics, including stochastic optimization, random
matrix theory, geometric functional analysis, randomized algorithms,
statistics, machine learning and compressed sensing.
A typical situation where the concentration inequalities are useful is the case where  one is interested in probabilistic bounds for a random variable which is the solution of a 
(stochastic) optimization problem.
This type of problems appear in statistics and stochastic optimisation. Many statistical estimators (e.g. the maximum-likelihood estimator) are solutions to random optimization problems. There is a substantial statistical literature dealing with concentration in statistics, see \citet{M} for an overview. On the stochastic optimisation side let us mention the bin packing problem and the travelling salesman problem where the concentration approach leads to  rather sharp probabilistic bounds for the quantities of interest. 
For example, in the bin packing problem  we are given \( n \) items of sizes in the interval \( [0,1] \) and are required to pack them into the fewest number of unit-capacity bins as possible. In the stochastic version, the item sizes are independent random variables in the interval \( [0, 1]. \)
\par
In this note we prove rather general and sharp concentration inequality for  smooth random fields. As a simple corollary of the main result we get a sharp exponential inequality for  
a convex function of the maximal eigenvalue of a random matrix.

\section{Main setup}

Let \( G(x;\thetav), \) \( \thetav\in \Theta \subseteq \R^{\dimp} \) be a family of real valued functions on \( \R^n \) and let \( X \) be a random vector in \( \R^n. \)  The purpose of this paper is to derive  exponential probability bounds for the random variable:
\[
\sup_{\thetav \in \Theta} G(X,\thetav). 
\] 
First we make the following assumptions.
\begin{description}
\item[\( \bb{(G\!C)} \)] 
    \textit{The function \( G(x,\thetav) \) is smooth in \( \thetav \) for any \( x\in \R^n \) and the mean function \( M(\thetav)\eqdef\E G(X,\thetav) \) is three times continuously differentiable in \( \thetav. \)  
Denote
\begin{EQA}[c]
    \thetavs
    \eqdef
    \argmax_{\thetav \in \Theta} M(\thetav).
\label{thetavsd}
\end{EQA}
There is a positive definite symmetric matrix \( D^* \) and a positive number \( \rr_0>0 \) such that} 
\begin{EQA}[c]
\nabla^2 M(\thetav)\preceq -D^*, \quad |\thetav-\thetavs|=\rr, \quad\rr>\rr_0.
\label{glob_conc}
\end{EQA}
\end{description}
\begin{description}
\item[\( \bb{(V\!I)} \)] 
\textit{There is  a symmetric positive definite matrix \( V_0 \)  such that}
\begin{EQA}[c]
 \Var\bigl\{ \nabla_{\thetav} G(X,\thetavs) \bigr\}\preceq \VPc^{2}
\end{EQA}
and 
\begin{EQA}[c]
V^2_0\succeq \varepsilon^{-2} I, \quad \fis^{2} \DPc^{2} \succeq \VPc^{2}  
\label{AssId}
\end{EQA}
with \( \nabla^2 M(\thetavs)=-D^2_0, \) a small parameter \(  \varepsilon\in (0,1/2)\) and \( \fis\in \R. \)
\end{description}
Introduce a centred random field:
\begin{EQA}[c]
    \zeta(\thetav)
    \eqdef
     G(X;\thetav)-\E G(X;\thetav)
\label{thetavsd}
\end{EQA}
and
a local elliptic 
neighbourhood of \( \thetavs \) via
\begin{EQA}[c]
    \Thetas(\rr) 
    \eqdef 
    \{ \thetav \in \Theta: \| \VPc (\thetav - \thetavs) \| \le \rr \}.
\label{Theta0R}
\end{EQA}   
Finally we make two integrability assumptions. 
\begin{comment}
\begin{description}
\item[\( \bb{(E\!D_{0})} \)] 
   \emph{ There are constants \( \gmb > 0 \), \( \nunu \ge 1 \) such that 
     for all 
    \( |\lambda| \le \gmb \) 
    } 
\begin{EQA}[c]
\label{expzetacloc} 
\sup_{\gammav \in \R^{\dimp}} %\sup_{\thetav \in \Thetas(\rr)} 
    \log \E \exp\biggl\{ 
        \lambda \frac{\gammav^{\T} \nabla_{\thetav} G(X,\thetavs)}
                     {\| \VPc \gammav \|} 
    \biggr\} \le 
    \nunu^{2} \lambda^{2} / 2. 
\end{EQA}
\end{description}
\end{comment}
\begin{description}
\item[\( \bb{(E\!D)} \)] \emph{There exists a 
    constant \( \rhor_0 \) such that it holds for all \( \thetav \in 
    \Thetas(\rr) \) } and all \( \rr \le \rups \), 
\begin{EQA}
\label{expzetacloc1} 
    \sup_{\gammav \in \R^{\dimp}} %\sup_{\thetav \in \Thetas(\rr)} 
    \log \E \exp\biggl\{ 
        \lambda 
        \frac{\gammav^{\T} \{\nabla\zeta(\thetav) - \nabla \zeta(\thetavs) \}}
             {\rhor_0\varepsilon\rr \| \VPc \gammav \|} 
    \biggr\} 
    & \le &
    \nunu^{2} \lambda^{2} / 2,
    \qquad 
    |\lambda| \le \gmb.
\end{EQA}
\end{description}
\begin{description}
\item[\( \bb{(E\rr)} \)] 
    \textit{ It holds for any \( \lambda>0, \)} 
\begin{EQA}[c]
\label{expzetag} 
    \sup_{\gammav \in \R^{\dimp}} \,\,
    \log \E \exp\biggl\{ 
        \lambda \frac{\gammav^{\T} \nabla\zeta(\thetav)}
                     {\| \VPc \gammav \|} 
    \biggr\} \le 
    \nunu^{2} \lambda^{2} / 2. 
\end{EQA}
\end{description}
\paragraph{Discussion}
Under \( (G\!C) \) the second order Taylor expansion of the function \( M(\thetav) \) at \( \thetavs \) gives 
\begin{EQA}[c]
M(\thetav)=M(\thetavs)+\frac{1}{2}(\thetav-\thetavs)^\top \nabla^2 M(\thetavs)(\thetav-\thetavs)+R(\thetav)
    \label{taylor1}
\end{EQA}
with 
\[ 
\frac{R(\thetav)}{\|\thetav-\thetavs\|^2}=O(\|\thetav-\thetavs\|), \quad  \|\thetav-\thetavs\|\to 0.
\] 
 Then under  \( (V\!I) \)
\begin{EQA}[c]
\label{LmgfquadEL}
        \biggl| 
            \frac{2(M(\thetav)-M(\thetavs))}{\| \DPc (\thetav - \thetavs) \|^{2}} + 
            1
        \biggr|
        \le 
        \delta_0\varepsilon\rr, \quad \thetav\in\Theta_0(\rr)
\end{EQA}  
for some \( \delta_0>0. \) The condition \eqref{glob_conc} basically means that \( M \) is globally concave and together with the  Taylor expansion 
\begin{EQA}[c]
M(\thetav)=M(\thetavs)+(\thetav-\thetavs)^\top \nabla^2 M(\thetavs+\alpha(\thetav-\thetavs))(\thetav-\thetavs), \quad \alpha\in (0,1)
    \label{taylor2}
\end{EQA}
gives
\begin{EQA}[c]
         M(\thetav)-M(\thetavs) \le 
    -\frac{\lambda_{\min}(D^*)}{\lambda_{\max}(V^2_0)}   \rr^{2}\eqdef-\gmi^*  \rr^{2}
\label{xxentrtt}
\end{EQA}  
if \( \| \thetav - \thetavs \|=\rr. \)
\section{Main result}
Define for \( \BB \eqdef \DPc^{-1} \VPc^{2} \DPc^{-1} \)
\begin{EQA}[c]
\dimA \eqdef \tr (\BB) , 
    \qquad 
        \vA^{2}
    \eqdef
    2 \tr(\BB^{2}),
    \qquad 
    \lambda_{0}
    \eqdef
    \| \BB \|_{\infty}
    =
    \lambda_{\max}\bigl( \BB \bigr).
\label{BBrdd}
\end{EQA}   
\begin{theorem} Under assumptions \( (G\!C), \) \( (V\!I), \) \((E\!D)\) and \((E\rr)\) 
\begin{EQA}[c]
\P\Bigl(\sup_{\thetav\in \Theta}G(X,\thetav)>G(X,\thetavs)+\lambda_0\dimA/2 +c\lambda_0(\vA\sqrt{x}+x)\Bigr)\leq e^{-x},    
\label{main_ineq}
\end{EQA}
for any \( x>0 \) satisfying \( \varepsilon \sqrt{(x+3p)}<1 \) and some constant \( c \) depending on  \( \nunu, \) \( \gmi^*, \) \( \delta_0 \) and \( \rhor_0  \) only. 
\end{theorem}
\paragraph{Applications (maximal eigenvalue)}
Let \( A=(a_{ij})_{i,j=1}^p \) be a Hermitian random matrix with a positive definite symmetric mean matrix \( \E A \)
and let
\begin{EQA}[c]
G(A,\thetav)\eqdef\thetav^{\top} A\thetav-f(\|\thetav\|^2), \quad \thetav\in \Theta
\label{GEx}
\end{EQA}
with \( \Theta=\{\thetav\in \R^p: |\thetav|<R\}\) for some large enough \( R>0 \) and a nonnegative monotone increasing smooth function \( f. \)  Let \( f^* \) be the Legendre transform of \( f, \) then
\begin{EQA}[c]
\sup_{\thetav\in \Theta}G(A,\thetav)= f^*(\lambda_{\max}(A)).
\label{app_opt}
\end{EQA}
Since \( M(\thetav)\eqdef\E G(A,\thetav) =\thetav^{\top} \E A\thetav-f(\|\thetav\|^2),\)
\begin{EQA}[c]
\sup_{\thetav\in \Theta}M(\thetav)=f^*(\lambda_{\max}(\E A))
\label{app_opt_e}
\end{EQA}
and the maximum is attained in the point \( \thetav^*=\sqrt{r^*}\bb{e}_p, \)
where \( \bb{e}_p \) is the eigenvector of the matrix \( \E A \) corresponding to its largest eigenvalue and \( r^*>0 \) solves the equation \( f'(r^*)=\lambda_{\max}(\E A). \) Moreover 
\begin{EQA}[c]
\nabla^2 M(\thetav)=\E A-f'(\|\thetav\|^2)\,I-f''(\|\thetav\|^2)\,\thetav\thetav^{\top}
\label{app_hess}
\end{EQA}
and as a result 
\begin{EQA}[c]
\nabla^2 M(\thetavs)=\E A-\lambda_{\max}(\E A)\,I-f''(r^*)\,r^*\bb{e}_p\bb{e}^{\top}_p\eqdef-D^2_0 \label{app_hess}
\end{EQA}
for some positive definite matrix \( D_0, \) provided \( f''(r^*)>0. \) Hence the assumption (GC) is fulfilled if \( f \) is globally convex. Assume
\begin{EQA}[c]
\sup_{\|\thetav\|=r^*}\sup_{\gammav \in \R^{\dimp}} \,\,
    \log \E \exp\biggl\{ 
        \lambda \frac{\gammav^{\T} (A-\E A)\thetav}
                     {\| \VPc \gammav \|} 
    \biggr\} \le 
    \nunu^{2} \lambda^{2} / 2, 
\label{app_em}
\end{EQA}
where   \( \VPc= \Var(A\thetavs)\).
Our main result implies 
\begin{EQA}
\P\Bigl(f^*(\lambda_{\max}(A))&-&f^*(\lambda_{\max}(\E A))\geq 
\\
&&\frac{\lambda_0 \dimA}{2}+(\thetavs)^{\top} (A-\E A) \thetavs+c\lambda_0(\vA\sqrt{x}+x)\Bigr)\leq e^{-x}
\label{eq:app1}
\end{EQA}
with \( \dimA=\tr\bigl(D_0^{-2} \VPc^{2}\bigr) \) and \( \vA^2= \tr\bigl(D_0^{-4} \VPc^{4}\bigr).\)
Furthermore it follows from \eqref{app_em}
\begin{EQA}[c]
\P\Bigl((\thetavs)^{\top} (A-\E A) \thetavs>\sqrt{x}\| \VPc \thetavs \|\Bigr)\leq e^{\nu_0^2/2}e^{-x}.
\label{eq:app2}
\end{EQA}
Combining \eqref{eq:app1} with \eqref{eq:app2}, we get
\begin{EQA}
\P\Bigl(f^*(\lambda_{\max}(A))&-&f^*(\lambda_{\max}(\E A))\geq 
\\
&&\frac{\lambda_0 \dimA}{2}+\sqrt{x}\| \VPc \thetavs \|+c\lambda_0(\vA\sqrt{x}+x)\Bigr)\leq (1+e^{\nu_0^2/2})\,e^{-x}.
\label{eq:appmain}
\end{EQA}
Let us compare the above inequality with the known results on the maximal eigenvalue of a  random Hermitian matrix.
For example, in \citet{MJCFT} an exponential inequality for the spectral norm of a bounded Hermitian random matrix \( A \) is derived via the method of exchangeable pairs.  In particular, it is shown that if \( A=X_1+\ldots+X_n, \) where \( X_1,\ldots,X_n \) are independent identically distributed Hermitian \( p\times p \) matrices satisfying
\begin{EQA}[c]
X_k^2\preceq B^2, \quad k=1,\ldots,n,
\label{cond}
\end{EQA}
then
\begin{EQA}[c]
\P(\lambda_{\max}(A-\E A)>t)\leq p\cdot \exp \left(-t^2/2\sigma^2\right)
\label{spnorm_ineq}
\end{EQA}
with \( \sigma^2=\frac{n}{2}\left\|B^2+\Var(X_1)\right\|. \)
The inequality \eqref{spnorm_ineq} is in fact equivalent to the following one
\begin{EQA}[c]
\P\Bigl(\lambda_{\max}(A-\E A)>\sqrt{2(x+\log p)}\sigma\Bigr)\leq \exp \left(-x\right)
\label{spnorm_ineq1}
\end{EQA}
In our setting with \( f(x)=nx^2 \) we get  \( r^*= \lambda_{\max}(\E A)/(2n)=\lambda_{\max}(\E X_1)/2,\)
\begin{eqnarray*}
V^2_0&=&\lambda_{\max}(\E A)\Var(A\bb{e}_p)/(2n)
\\ 
&=& n\cdot\lambda_{\max}(\E X_1)\Var(X_1\bb{e}_p)/2
\\
D^2_0&=&n\cdot\lambda_{\max}(\E A)(I+2\bb{e}_p\bb{e}_p^\top-\E X_1/\lambda_{\max}(\E X_1)) 
\end{eqnarray*}
and 
\begin{EQA}
D^{-2}_0V_0^2&=& (I+2\bb{e}_p\bb{e}_p^\top-\E X_1/\lambda_{\max}(\E X_1))^{-1}\Var(X_1\bb{e}_p)/2.
\label{}
\end{EQA}
Hence \( \dimA=c_1\cdot p \) and \( \vA=c_2\cdot p \) for some constants \( c_1 \) and \( c_2 \) not depending on \( n \) and \( p. \)
Furthermore,
\begin{EQA}[c]
\| \VPc \thetavs \|=\sqrt{n}\cdot\lambda_{\max}(\E X_1)\left\|\bb{e}_p^\top\Var^{1/2}(X_1\bb{e}_p)\right\|.
\label{VPc}
\end{EQA}
and the inequality  \eqref{eq:appmain} transforms to 
\begin{EQA}
\P\Bigl(\lambda^2_{\max}(A)-\lambda^2_{\max}(\E A)&\geq& c(\sqrt{nx}+x+p)\Bigr)\leq (1+e^{\nu_0^2/2})\,e^{-x}
\label{eq:appmain2}
\end{EQA} 
with some constant \( c>0 \) not depending on \( p \) and \( n. \)
\par
Note that in the domain \( \lambda_{\max}(A+\E A)>1, \) \( p/n<1 \) the inequality \eqref{eq:appmain2} is more accurate than \eqref{spnorm_ineq1}.  Moreover, while the condition \eqref{cond} basically means that \( A \)
is bounded with probability \( 1, \) our assumption \eqref{app_em} only requires a sub-gaussian behaviour
of \( A-\E A. \) 
\section{Proof of the main result}
\begin{proof}
Denote \( Z(\thetav,\thetavs)\eqdef G(X,\thetav)-G(X,\thetavs). \) We get from Proposition~\ref{TapproxLL}, Lemma~\ref{LLbreveloc} and Lemma~\ref{Lxivgap}
\begin{EQA}
 \sup_{\thetav \in \Thetas(\rr)}Z(\thetav,\thetavs)&\leq &\sup_{\thetav \in \Theta}\Zab(\thetav,\thetavs)+\errb(\rr) 
 \\
 & \leq & \| \xivb \|^{2}/2+\errb(\rr)  
 \\
 & = & \| \xiv \|^{2}/2+\bigl\{\| \xivb \|^{2}-\| \xiv \|^{2}\bigr\}/2+\errb(\rr) 
 \\
 &\leq & \| \xiv \|^{2}/2+\frac{\tau_{\rd}}{2(1 - \tau_{\rd})} \| \xiv \|^{2}+\errb(\rr)
 \\
 &=& \frac{\| \xiv \|^{2}}{2(1 - \tau_{\rd})} +\errb(\rr)
\end{EQA}
Now Proposition~\ref{LLbrevelocm} implies
\begin{EQA}
 \P\Bigl(\sup_{\thetav \in \Thetas(\rr_0)}Z(\thetav,\thetavs)&>&\frac{\lambda_0\cdot\zz(\xx,\BB)}{2(1 - \tau_{\rd})}+6\nu_0\omega_0\varepsilon\rr_0\bigl( 1 + \sqrt{\xx + 3p} \bigr)^{2}\Bigr)\leq  4e^{-\xx},
 \end{EQA}
where \( \zz(\xx,\BB) \) is given by \eqref{zzxxppdBl}.
Next, we shall prove that there is  \( \rr_0>0 \) and a deterministic upper function
\( \pnnd(\thetav)\geq 0 \) such that
\begin{EQA}[c]
    \P\Bigl( 
        \sup_{\thetav \in \Theta\setminus \Thetas(\rr)} 
            \bigl\{ Z(\thetav,\thetavs) + \pnnd(\thetav) \bigr\}
        \ge  
        0
    \Bigr)
    \le 
    \ex^{-\xx} 
\label{hitprobxxgl}
\end{EQA}
for \( \rr>\rr_0 \) and \( \xx > 0 \).
The inequality \eqref{hitprobxxgl} then implies 
\begin{EQA}[c]
    \P\Bigl( \sup_{\thetav \not\in \Thetas(\rr_0)}Z(\thetav,\thetavs)\geq 0\Bigr) 
    \le 
    \ex^{-\xx} .
\label{PnotinTsruc}
\end{EQA}    
A possible way of checking the condition \eqref{hitprobxxgl} is based on a lower 
quadratic bound for the negative expectation \( M(\thetav) \)
in the sense of condition \eqref{xxentrtt}.
\begin{lemma}
\label{CThittingglrc}
Suppose \( (GC) \) and \( (E\rr). \)
Let, for \( \rr \ge \rups \),
\begin{EQA}
\label{cgmi1rrc}
    6 \nunu \sqrt{\xx + 3p}
    & \le &
    \rr \gmi^* ,
\label{cgmi2rrc}
\end{EQA}
with \( \xx + 3p \ge 2.5. \)  Then
\begin{EQA}[c]
    \P\Bigl( \sup_{\thetav \not\in \Thetas(\rr)}Z(\thetav,\thetavs)\geq 0\Bigr) 
    \le 
    \ex^{-\xx} .
\label{PnotinTsruc}
\end{EQA}    
\end{lemma}
\begin{proof}
The result follows from Theorem~\ref{Thitting} with \( \mubc = \frac{\gmi^*}{3 \nunu} \),
\( \pen(\mubc) \equiv 0 \),
\( \UP(\thetav) = Z(\thetav,\thetavs) - \E Z(\thetav,\thetavs) \) and 
\( \Ldrift(\thetav,\thetavs) = M(\thetav)-M(\thetavs) 
\ge \frac{\gmi^*}{2} \| \VPc (\thetav - \thetavs) \|^{2} \).
\end{proof}
It follows now from Lemma~\ref{CThittingglrc} that the inequality 
\begin{EQA}
\sup_{\thetav \in \Theta}Z(\thetav,\thetavs)&\leq&\sup_{\thetav \in \Thetas(\rr_0)}Z(\thetav,\thetavs)
\end{EQA}
holds with probability at least \( 1-\ex^{-\xx}. \)
As a result we get the desired inequality.
 \end{proof}
\subsection{Auxiliary results}
Let \( \rddelta, \rdomega \) be nonnegative constants.
Introduce for a vector \( \rd = (\rddelta,\rdomega) \) the following notation:
\begin{EQA}
    \Zab(\thetav,\thetavs)
    & \eqdef &
    (\thetav - \thetavs)^{\T} \nabla \zeta(\thetavs)
    - \| \DPb (\thetav - \thetavs) \|^{2}/2 
    \\
    &=&
    \xivb^{\T} \DPb (\thetav - \thetavs) 
    - \| \DPb (\thetav - \thetavs) \|^{2}/2 ,
\label{bLquadloc}
\end{EQA}
where \( \nabla \zeta(\thetavs) = \nabla_{\thetav} G(X,\thetavs) \) by \( \nabla M(\thetavs) = 0 \)
and
\begin{EQA}
    \DPb^{2} 
    &=&
    \DPc^{2} (1 - \rddelta) - \rdomega \VPc^{2}, 
    \qquad 
    \xivb
    \eqdef 
    \DPb^{-1} \nabla G(X,\thetavs) .
\label{xivalpsn}
\end{EQA}  
Here we implicitly assume that with the proposed choice of the constants 
\( \rddelta \) and \( \rdomega \),
the matrix \( \DPb^{2} \) is non-negative: \( \DPb^{2} \ge 0 \). 
The representation \eqref{bLquadloc} indicates that the process 
\( \Zab(\thetav,\thetavs) \) has the quadratic local structure.
Now, given \( \rr \), fix some \( \rddelta \ge \rddelta_0 \varepsilon\rr \) and
\( \rdomega \ge 3\nunu \rhor_0\varepsilon\rr \) with the value \( \rddelta_0 \) from 
 \eqref{LmgfquadEL} and \( \rhor_0\) from condition \( (E\!D) \).
Finally set \( \rdm = - \rdb \), so that 
\(
    \DPm^{2} 
    =
    \DPc^{2} (1 + \rddelta) + \rdomega \VPc^{2}.
%\label{DPm2do}
\)  
\begin{proposition}
\label{TapproxLL}
Assume \( (E\!D) \) and \( (V\!I) \). 
Let for some \( \rr \), the values 
\( \rdomega \ge 3 \nunu \, \rhor_0 \varepsilon\rr \) and \( \rddelta \ge \rddelta_0 \varepsilon\rr \) be such that 
\( \DPc^{2} (1 - \rddelta) - \rdomega \VPc^{2} \ge 0 \). 
%Set \( \rdb = (\rddelta,\rdomega) \), \( \rdm = - \rdb = (-\rddelta,-\rdomega) \).
Then
\begin{EQA}[c]
    \Zam(\thetav,\thetavs) - \errm(\rr)
    \le 
    Z(\thetav,\thetavs) 
    \le 
    \Zab(\thetav,\thetavs) + \errb(\rr),
    \quad 
    \thetav \in \Thetas(\rr),
    \qquad 
\label{LttbLtt}
\end{EQA}    
with  \( \Zab(\thetav,\thetavs), \Zam(\thetav,\thetavs) \) defined by 
\eqref{bLquadloc}. 
The error terms \( \errb(\rr) \) and \( \errm(\rr) \) satisfy 
\begin{EQA}[c]
    \P\bigl( 
        \rdomega^{-1} \max\{\errb(\rr),\errm(\rr)\}
        \ge 
        \bigl( 1 + \sqrt{\xx + 3p} \bigr)^{2}
    \bigr)
    \le 
    \exp\bigl( - \xx \bigr).
\label{errbzrr}
\end{EQA}    
\end{proposition}
\begin{proof}
Consider for fixed \( \rr \) and \( \rdb = (\rddelta,\rdomega) \) the quantity
\begin{EQA}
\label{Delta1loc}
    \errb(\rr)
    & \eqdef &
    \sup_{\thetav \in \Thetas(\rr)}
    \bigl\{ 
        Z(\thetav,\thetavs) - \E Z(\thetav,\thetavs) 
        - (\thetav - \thetavs)^{\T} \nabla Z(\thetav,\thetavs) 
        - \frac{\rdomega}{2} \| \VPc (\thetav - \thetavs) \|^{2} 
    \bigr\} .
\end{EQA}
As \( \rddelta \ge \rddelta_0\varepsilon\rr \), it holds 
\( - M(\thetav) \ge (1 - \rddelta) \DPc^{2} \) and 
\( Z(\thetav,\thetavs) - \Zab(\thetav,\thetavs) \le \errb(\rr) \).
Moreover, in view of \( \nabla M(\thetavs) = 0 \), the definition of 
\( \errb(\rr) \) can be rewritten as
\begin{EQA}
\label{Delta1loc}
    \errb(\rr)
    & = &
    \sup_{\thetav \in \Thetas(\rr)}
    \bigl\{ 
        \zeta(\thetav)-\zeta(\thetavs)- (\thetav - \thetavs)^{\T} \nabla \zeta(\thetavs) 
        - \frac{\rdomega}{2} \| \VPc (\thetav - \thetavs) \|^{2} 
    \bigr\} .
\end{EQA}    
Now the claim of the theorem can be easily reduced to an exponential bound for 
the quantity \( \errb(\rr) \).
We apply Theorem~\ref{Tsmoothpenlc} to the process 
\begin{EQA}[c]
    \UP(\thetav,\thetavs) 
    = 
    \frac{1}{\rhor_0\varepsilon\rr}
    \bigl\{ 
        \zeta(\thetav)-\zeta(\thetavs) - (\thetav - \thetavs)^{\T} \nabla \zeta(\thetavs) 
    \bigr\},
    \qquad 
    \thetav \in \Thetas(\rr),
\label{UPloce}
\end{EQA}
and \( \VVc = \VPc \).
Condition \( (\CS\! D) \) follows from 
\( (E\!D) \) with the same \( \nunu \) and \( \gmb \) in view of 
\( \nabla \UP(\thetav,\thetavs) 
= \bigl\{ \nabla \zeta(\thetav) - \nabla \zeta(\thetavs) \bigr\} / \rhor\varepsilon\rr \).
So, the conditions of Theorem~\ref{Tsmoothpenlc} are fulfilled yielding 
\eqref{errbzrr} in view of \( \rdomega \ge 3 \nunu \, \rhor_0 \varepsilon\rr \). 
\end{proof}
\begin{lemma}
\label{LLbreveloc}
It holds 
\begin{EQA}
    \sup_{\thetav \in \Theta} \Zab(\thetav,\thetavs)\leq\sup_{\thetav \in \R^p} \Zab(\thetav,\thetavs)
    &=&
    \| \xivb \|^{2}/2 
\label{supLat}
\end{EQA}    
\end{lemma}
\begin{lemma}
\label{Lxivgap}
Define \( \xiv \eqdef \DPc^{-1} \nabla \zeta(\thetavs). \) Suppose \( (V\!I) \) and let 
\( \tau_{\rd} \eqdef \varepsilon r_0(\rddelta_0 + 3\nunu \rhor_0\fis^{2}) < 1 \). 
Then
\begin{EQA}[c]
    \| \xivb \|^{2} - \| \xiv \|^{2}
    \le 
    \frac{\tau_{\rd}}{1 - \tau_{\rd}} \| \xiv \|^{2},
    \quad 
    \| \xiv \|^{2} - \| \xivm \|^{2}
    \le 
    \frac{\tau_{\rd}}{1 + \tau_{\rd}} \| \xiv \|^{2}.
\label{xivbccmpa}
\end{EQA}    
\end{lemma}

\begin{proposition}
\label{LLbrevelocm}
Let \( (E\!D) \) hold with \( \nunu = 1. \)
Then
\( \E \| \xiv \|^{2} \le \dimA \), and for each \( \xx>0 \)
\begin{EQA}
    \P\bigl( \| \xiv \|^{2} \ge \lambda_0\cdot\zz(\xx,\BB) \bigr)
    & \le &
    2 \ex^{-\xx},
\label{PxivbzzBB}
\end{EQA}    
where \( \zz(\xx,\BB) \) is defined by
\begin{EQA}
\label{PzzxxpB}
    \zz(\xx,\BB)
    & \eqdef &
    \begin{cases}
        \dimA+2 \vA \xx^{1/2}, &  \xx \le \vA/18 , \\
        \dimA+6 \xx & \vA/18 < \xx. \\
%        \bigl| \yyc + 2 (\xx - \xxc)/\gmc \bigr|^{2}, & \xx > \xxc .
    \end{cases}
\label{zzxxppdBl}
\end{EQA}    
\end{proposition}
\begin{proof}
It holds
\begin{EQA}
    \E \| \xiv \|^{2}
    &=&
    \E \tr  \xiv \xiv^{\T}
    \\
    &=&
    \tr \DPc^{-1} \bigl[ \E \nabla \zeta(\thetavs) \{ \nabla \zeta(\thetavs) \}^{\T} \bigr] 
    \DPc^{-1}
    =
    \tr \bigl[ \DPc^{-2} \Var \bigl\{ \nabla \zeta(\thetavs) \bigr\} \bigr]
\label{Exiv2tr}
\end{EQA}    
and  
\( \gammav^{\T} \Var \bigl\{ \nabla \zeta(\thetavs) \bigr\} \gammav 
\le \gammav^{\T} \VPc^{2} \gammav \) and thus, 
\( \E \| \xiv \|^{2} \le \dimA \).
The deviation bound \eqref{PxivbzzBB} is proved in Corollary~\ref{CTxivqLDAB}.
\end{proof}
\section{Appendix}
The proofs of the results below can be found in Appendix A and Appendix B of \citet{S}.
\subsection{Deviation probability for quadratic forms}
\label{Chgqform}
\label{Sprobabquad}
Assume that
\begin{EQA}[c]
    \log \E \exp\bigl( \gammav^{\T} \xiv \bigr) 
    \le 
    \| \gammav \|^{2}/2,
    \qquad 
    \gammav \in \R^{\dimp}, \, \| \gammav \| \le \gm .
\label{expgamgm}
\end{EQA}
This section presents a general exponential bound for the probability 
\( \P\bigl( \| \BB \xiv \| > \yy \bigr) \)  with a given matrix 
\( \BB \) and a vector \( \xiv \) obeying the condition \eqref{expgamgm}. 
We assume that \( \BB \) is symmetric. 
Define important characteristics of \( \BB \)
\begin{EQA}[c]
    \dimA = \tr (\BB^{2}) , 
    \qquad 
    \vA^{2} = 2 \tr(\BB^{4}),
    \qquad 
    \lambdaB \eqdef \| \BB^{2} \|_{\infty} \eqdef \lambda_{\max}(\BB^{2}) .
\label{dimAvAlb}
\end{EQA}
For simplicity of formulation we suppose that \( \lambdaB = 1 \),
otherwise one has to replace \( \dimA \) and \( \vA^{2} \) with 
\( \dimA/\lambdaB \) and \( \vA^{2}/\lambdaB \).
Let \( \gm \) be given in \eqref{expgamgm}.
Define 
\( \wwc \) by the equation
\begin{EQA}[c]
    \frac{\wwc(1+\wwc)}{(1+\wwc^{2})^{1/2}}
    =
    \gm \dimA^{-1/2} .
\label{wc212A}
\end{EQA}
%However, if \( \muc > 2/3 \), then cut by this value, i.e. set 
%\( \muc = (2/3) \wedge \{ \gm^{2}/(\dimA + \gm^{2}) \} \).
Define also \( \muc = \wwc^{2}/(1+\wwc^{2}) \wedge 2/3 \).
Note that \( \wwc^{2} \ge 2 \) implies \( \muc = 2/3 \).
Further define
\begin{EQA}
    \yyc^{2} = (1 + \wwc^{2}) \dimA,
    \qquad 
    2 \xxc
    & = &
%    \gmc \yyc + \log \det\{ \Id_{\dimp} - (\gmc/\yyc) \BB^{2} \} 
%    =
    \muc \yyc^{2} + \log \det\{ \Id_{\dimp} - \muc \BB^{2} \} .
\label{xxcyycA}
\end{EQA}
\begin{theorem}
\label{TxivqLDA}
Let a random vector \( \xiv \) in \( \R^{\dimp} \) fulfill \eqref{expgamgm}.
Then for each \( \xx < \xxc \)
\begin{EQA}
    \P\bigl( 
        \| \BB \xiv \|^{2}-\E \| \BB \xiv \|^{2}> (2 \vA \xx^{1/2}) \vee (6 \xx),
        \| \BB \xiv \| \le \yyc
    \bigr)
    & \le &
    2 \exp( - \xx ) .
%    + 8.4 \exp( - \xxc ) 
\label{expxiboA}
\end{EQA}    
%with \( a_{c} \eqdef 6 \vee (4 \muc^{-1}) \).
Moreover, for \( \yy \ge \yyc \), with
\( \gmc = \gm - \sqrt{\muc \dimA} = \gm \wwc/(1+\wwc) \),
it holds
\begin{EQA}
    \P\bigl( \| \BB \xiv \| > \yy \bigr)
    & \le &
%    8.4 \Bigl( \frac{\yy}{\yy - \yyn} \Bigr)^{\dimp/2}
%    \exp\bigl( - \gm \yy/2 \bigr) 
%    \le 
    8.4 \exp\bigl( - \xxc - \gmc (\yy - \yyc)/2 \bigr) .
\label{expxibogA}
\end{EQA}
\end{theorem}

Let us now describe the value \( \zz(\xx,\BB) \) ensuring a small value for the large deviation probability 
\( \P\bigl( \| \BB \xiv \|^{2} > \zz(\xx,\BB) \bigr) \).
For ease of formulation, we suppose that \( \gm^{2} \ge 2 \dimA \) yielding 
\( \muc^{-1} \le 3/2 \).
The other case can be easily adjusted. 

\begin{corollary}
\label{CTxivqLDAB}
Let \( \xiv \) fulfill \eqref{expgamgm} with \( \gm^{2} \ge 2 \dimA \). 
Then it holds for \( \xx \le \xxc \) with \( \xxc \) from \eqref{xxcyycA}:
\begin{EQA}
\label{PzzxxpB}
    \P\bigl( \| \BB \xiv \|^{2}-\E \| \BB \xiv \|^{2} \ge \zz(\xx,\BB) \bigr)
    & \le &
    2 \ex^{-\xx} + 8.4 \ex^{-\xxc},
    \\
    \zz(\xx,\BB)
    & \eqdef &
    \begin{cases}
        2 \vA \xx^{1/2}, &  \xx \le \vA/18 , \\
        6 \xx & \vA/18 < \xx \le \xxc .
    \end{cases}
\label{zzxxppdB}
\end{EQA}    
For \( \xx > \xxc \)
\begin{EQA}
    \P\bigl( \| \BB \xiv \|^{2} \ge \zzc(\xx,\BB) \bigr)
    & \le &
    8.4 \ex^{-\xx},
    \qquad 
    \zzc(\xx,\BB)
    \eqdef 
    \bigl| \yyc + 2 (\xx - \xxc)/\gmc \bigr|^{2} .
\label{zzcxxppdB}
\end{EQA}    
\end{corollary}

\subsection{Some results for empirical processes}
This chapter presents some general results of the theory of empirical processes. We assume some exponential moment conditions on the increments of the process which allows to apply the well developed chaining arguments in Orlicz spaces;  We, however, follow the more recent approach inspired by the notions of generic chaining and majorizing measures due to M. Talagrand; see e.g. \citet{T}. The results are close to that of \citet{B}. We state the results in a slightly different form and present an independent and self-contained proof.
The first result states a bound for local fluctuations of the process 
\( \UP(\ups) \) given on a metric space \( \Ups \).
% in a small vicinity of any fixed point \( \upsd \). 
Then this result will be used for bounding the maximum of the negatively drifted 
process \( \UP(\ups) - \UP(\upss) - \rho \dist^{2}(\ups,\upss) \) over a  
vicinity \( \Upss(\rr) \) of the central point \( \upss \).
The behavior of \( \UP(\ups) \) outside of the local central set \( \Upss(\rr) \) is 
described using the \emph{upper function} method. 
Namely, we construct a multiscale deterministic function \( \up(\mubc,\ups) \) ensuring that 
with probability at least \( 1 - \ex^{-\xx} \) it holds
\( \mubc \UP(\ups) + \up(\mubc,\ups) \le \zz(\xx) \) for all \( \ups \not\in \Upss(\rr) \)
and \( \mubc \in \Mubc \), where \( \zz(\xx) \) grows linearly in \( \xx \). 
\par
Let \( \dist(\ups,\upsc) \) be a semi-distance on \( \Ups \).
% usually caller the intrinsic semi-metric. 
We suppose the following condition to hold: 

\begin{description}
\item[\( \bb{(\CS{d})} \)]
    \textit{
    There exist \( \gmb > 0 \), \( \rups > 0 \), 
    \( \nunu \ge 1 \), 
    such that for any  \( \lambda \le \gmb \) and \( \ups,\upsc \in \Ups \)
    with \( \dist(\ups,\upsc) \le \rups \)
    }
\begin{EQA}[c]
\label{ExpboundUP}
    \log \E \exp \biggl\{ 
        \lambda \frac{\UP(\ups) - \UP(\upsc)}{\dist(\ups,\upsc)} 
    \biggr\}
    \le  
    \nunu^{2} \lambda^{2}/2 .
\end{EQA}
\end{description}

Formulation of the result involves a sigma-finite measure \( \mes \) on the space 
\( \Ups \) which is often 
called the \emph{majorizing measure} and used in the \emph{generic chaining} device; 
A typical example of choosing \( \mes \) is the Lebesgue measure on \( \R^{\dimp} \).
Let \( \Upsd \) be a subset of \( \Ups \), 
a sequence \( \rr_{k} \) be fixed with \( \rr_{0} = \diam(\Upsd) \)
and \( \rr_{k} = \rr_{0} 2^{-k} \).
Let also \( \B_{k}(\ups) \eqdef \{ \upsc \in \Upsd: \dist(\ups,\upsc) \le \rr_{k} \} \) 
be the \( \dist \)-ball centered at \( \ups \) of radius 
\( \rr_{k} \) and \( \mes_{k}(\ups) \) denote its \( \mes \)-measure:
\begin{EQA}[c]
    \mes_{k}(\ups)
    \eqdef
    \int_{\B_{k}(\ups)} \mes(d\upsc) 
    =
    \int_{\Upsd} \Ind\bigl( \dist(\ups,\upsc) \le \rr_{k} \bigr) \mes(d\upsc).
\label{meskups}
\end{EQA}  
Denote also %\( \mes_{0} \eqdef \mes(\Upsd) \) and  
\begin{EQA}[c]
    M_{k}
    \eqdef
    \max_{\ups \in \Upsd} \frac{\mes(\Upsd)}{\mes_{k}(\ups)} 
    \qquad 
    k \ge 1.
\label{MkUps}
\end{EQA}    
%Obviously \( M_{0} = 1 \).
Finally set
\( c_{1} = 1/3 \), \( c_{k} = 2^{-k+2}/3 \) for \( k \ge 2 \), and define
the value \( \entrl(\Upsd) \) by 
\begin{EQA}[c]
    \entrl(\Upsd)
    \eqdef
    \sum_{k=1}^{\infty} c_{k} \log(2 M_{k})
    =
    \frac{1}{3} \log(2M_{1}) + \frac{4}{3} \sum_{k=2}^{\infty} 2^{-k} \log(2 M_{k}) .
\label{entrldef}
\end{EQA}    

\begin{theorem}
\label{TUPUpsd}
Let \( \UP \) be a separable process following to \( (\CS{d}) \). 
If \( \Upsd \) is a \( \dist \)-ball in \( \Ups \) with the center \( \upsd \) and the radius \( \rups \), 
i.e. \( \dist(\ups,\upsd) \le \rups \) for all \( \ups \in \Upsd \), then for
\( \lambda \le \gmd \eqdef \nunu \gmb \)
\begin{EQA}[c]
\label{Upsdboundd}
    \log \E \exp \Bigl\{ \frac{\lambda}{3 \nunu \rups}
        \sup_{\ups \in \Upsd} \bigl| \UP(\ups) - \UP(\upsd) \bigr| 
    \Bigr\}
    \le 
    \lambda^{2}/2 + \entrl(\Upsd).
\end{EQA}
\end{theorem}
Due to the result of Theorem~\ref{TUPUpsd}, the bound for the maximum of 
\( \UP(\ups,\upss) \) over \( \ups \in \B_{\rr}(\upss) \) grows quadratically in \( \rr \). 
So, its applications to situations with \( \rr^{2} \gg \entrl(\Upsd) \) are limited.
The next result shows that introducing a negative quadratic drift helps to
state a uniform in \( \rr \) local probability bound.
Namely, the bound for the process 
\( \UP(\ups,\upss) - \rho \dist^{2}(\ups,\upss)/2 \) with some 
positive \( \rho \) over a ball \( \B_{\rr}(\upss) \) around 
the point \( \upss \) only depends on the drift coefficient \( \rho \) but 
not on \( \rr \).
\begin{theorem}
\label{TsuprUP}
Let \( \rrb \) be such that \( (\CS{d}) \) holds on \( \B_{\rrb}(\upss) \).
Let also \( \entrl(\Upsd) \le \entrlb \) for \( \Upsd = \B_{\rr}(\upss) \) 
with \( \rr \le \rrb \).
If \( \rho > 0 \) and \( \zz \) are fixed to ensure 
\( \sqrt{2 \rho \zz} \le \gmd = \nunu \gmb \) and \( \rho (\zz - 1) \ge 2 \), then 
it holds
\begin{EQA}
    && \nquad
    \log \P\biggl( 
        \sup_{\ups \in \B_{\rrb}(\upss)} 
        \biggl\{ 
            \frac{1}{3 \nunu} \UP(\ups,\upss) - \frac{\rho}{2} \dist^{2}(\ups,\upss)
        \biggr\} > \zz 
    \biggr)
    \\
    & \le &
    - \rho (\zz - 1) + \log(4 \zz) + \entrlb .
\label{Psuprzz}
\end{EQA}    
Moreover, if \( \sqrt{2 \rho \zz} > \gmd \), then 
%with \( \lambda = \nunu \gmb \)
\begin{EQA}
    && \nquad
    \log \P\biggl( 
        \sup_{\ups \in \B_{\rrb}(\upss)} 
        \biggl\{ 
            \frac{1}{3 \nunu} \UP(\ups,\upss) - \frac{\rho}{2} \dist^{2}(\ups,\upss) 
        \biggr\} > \zz 
    \biggr)
    \\
    & \le &
    - \gmd \sqrt{\rho (\zz - 1)} + \gmd^{2}/2 
    + \log(4 \zz) + \entrlb .
\label{PsuprzzLD}
\end{EQA}
\end{theorem}
This result can be used for describing the concentration bound for 
the maximum of \( (3\nunu)^{-1} \UP(\ups,\upss) - \rho \dist^{2}(\ups,\upss)/2 \).
Namely, it suffices to find \( \zz \) ensuring the prescribed 
deviation probability. 
We state the result for a special case with \( \rho = 1 \) and 
\( \gmd \ge 3 \) which simplifies the notation.

\begin{corollary}
\label{CTsuprUP}
Under the conditions of Theorem~\ref{TsuprUP}, 
for any \( \xx \ge 0 \) with \( \xx + \entrlb \ge 4 \)
\begin{EQA}[c]
    \P\biggl( 
        \sup_{\ups \in \B_{\rrb}(\upss)} 
        \Bigl\{ 
            \frac{1}{3 \nunu} \UP(\ups,\upss) - \frac{1}{2} \dist^{2}(\ups,\upss) 
        \Bigr\} 
        >
        \zzQ(\xx,\entrlb)
    \biggr)
    \le 
    \exp\bigl( - \xx \bigr) ,
\label{PUPxx}
\end{EQA}    
where with \( \gmd = \nunu \gmb \ge 2 \)
\begin{EQA}[c]
    \zzQ(\xx,\entrlb)
    \eqdef
    \begin{cases}
    \bigl( 1 + \sqrt{\xx + \entrlb} \bigr)^{2}  
        & \text{if } 1 + \sqrt{\xx + \entrlb} \le \gmd, \\
    1 + \bigl\{ 2 \gmd^{-1} (\xx + \entrlb) + \gmd \bigr\}^{2} 
        & \text{otherwise} .
  \end{cases}
\label{PUPxxl}
\end{EQA}
\end{corollary}
Let us now discuss the special case when \( \Ups \) is an open subset in 
\( \R^{\dimp} \),
the stochastic process \( \UP(\ups) \) is absolutely continuous and its gradient
\( \nabla \UP(\ups) \eqdef d \UP(\ups) / d \ups \) 
has bounded exponential moments.

\begin{description}
\item[\( \bb{(\CS\! D)} \)]\textit{
There exist \( \gmb > 0 \), %\( \rupsb > 0 \), 
\( \nunu \ge 1 \), and for each \( \ups \in \Ups \),
a symmetric non-negative matrix \( \VV(\ups) \)
such that for any  \( \lambda \le \gmb \) 
%and each \( \ups \) with \( \| \VV\ups \| \le \rupsb \),
and any unit vector \( \gammav \in \R^{\dimp} \), it holds
}
\begin{EQA}[c]
%    \nunu^{-1} \lambda^{2}/2 
%    \le 
    \log \E \exp \Bigl\{
       \lambda 
       \frac{\gammav^{\T}\nabla \UP(\ups)}
            {\| \VV(\ups) \gamma \|}
    \Bigr\} 
    \le 
    \nunu^{2} \lambda^{2}/2 .
\end{EQA}
\end{description}
Consider the local sets of the elliptic form 
\( \Upss(\rr) \eqdef \{ \ups: \| \VVc(\ups - \upss) \| \le \rr \} \), 
where \( \VVc \) dominates \( \VV(\ups) \) on this set: \( \VV(\ups) \preceq \VVc \).

\begin{theorem}
\label{Tsmoothpenlc}
Let \( (\CS\! D) \) hold with some \( \gmb \) and a matrix \( \VV(\ups) \).
% and \( \rr \le \gmb / \qqq\). 
Suppose that \( \VV(\ups) \preceq \VVc \)  for all \( \ups \in \Upss(\rr) \).
Then 
\begin{EQA}
    \P \biggl( \sup_{\ups \in \Upss(\rr)} 
        \Bigl\{ 
            \frac{1}{3 \nunu}\UP(\ups,\upss) 
            - \frac{1}{2} \| \VVc(\ups - \upss) \|^{2}
        \Bigr\}
        \ge 
        \zzQ(\xx,\dimp)
    \biggr)
    & \le &
    \exp(-\xx)  ,
\label{expUUsmooth}
\end{EQA}
where \( \zzQ(\xx,\dimp) \) coincides with \( \zzQ(\xx,\entrlb) \) from 
\eqref{PUPxxl} for \( \entrlb = \cdimb \dimp \).
\end{theorem}
The  previous result can be explained as a local upper function for the 
process \( \UP(\cdot) \).
Indeed, in a vicinity \( \B_{\rrb}(\upss) \) of the central point \( \upss \), it holds 
\( (3 \nunu)^{-1} \UP(\ups,\upss) \le \dist^{2}(\ups,\upss)/2 + \zz \) 
with a probability  exponentially small in \( \zz \).
Now we extend this local result to the whole set \( \Ups \) 
using multiscaling arguments. 
For simplifying the notations assume that \( \UP(\upss) \equiv 0 \).
Then \( \UP(\ups,\upss) = \UP(\ups) \).
We say that \( \up(\mubc,\ups) \) is a \emph{multiscale upper function} for 
\( \mubc \UP(\cdot) \) on a subset \( \Upsd \) of \( \Ups \) if 
\begin{EQA}[c]
    \P\Bigl( \sup_{\mubc \in \Mubc} \,\, \sup_{\ups \in \Upsd} 
        \bigl\{ \mubc \UP(\ups) - \up(\mubc,\ups) \bigr\}  
        \ge 
        \zz(\xx)
    \Bigr)
    \le 
    \ex^{-\xx} ,
\label{upfuncgl}
\end{EQA}    
for some fixed function \( \zz(\xx) \).
An upper function can be used for describing the concentration sets of the point of 
maximum \( \tilde{\ups} = \argmax_{\ups \in \Upsd} \UP(\ups) \);
see Theorem~\ref{Thitting} below.

The desired global bound requires an extension of the local exponential moment condition 
\( (\CS{d}) \). 
Below we suppose that the pseudo-metric \( \dist(\ups,\upsc) \) is given on the whole 
set \( \Ups \).
For each \( \rr \) this metric defines the ball \( \Upss(\rr) \) by the constraint
\( \dist(\ups,\upss) \le \rr \). 
Below the condition \( (\CS{d}) \) is assumed to be fulfilled for any \( \rr \), however 
the constant \( \gm \) may be dependent of the radius \( \rr \).
\begin{description}
\item[\( \bb{(\CS\rr)} \)]
    \textit{
%    The pseudo-metric \( \dist(\ups,\upsc) \) is given on the whole set \( \Ups \) and 
    For any \( \rr \), there exists \( \gm(\rr) > 0 \) such that 
    \eqref{ExpboundUP} holds for all \( \ups, \upsc \in \Upss(\rr) \) and all 
    \( \lambda \le \gm(\rr) \)}.
\end{description}
Condition \( (\CS\rr) \) implies a similar condition for the scaled process 
\( \mubc \UP(\ups) \) with \( \gm = \mubc^{-1} \gm(\rr) \) and \( \dist(\ups,\upsc) \)
replaced by \( \mubc \dist(\ups,\upsc) \).
Corollary~\ref{CTsuprUP} implies for any \( \xx \) with 
\( 1 + \sqrt{\xx + \entrlb} \le \gmd(\rr) \eqdef \nunu \gm(\rr)/\mubc \)
\begin{EQA}[c]
    \P\biggl( 
        \sup_{\ups \in \B_{\rr}(\upss)} 
        \Bigl\{ 
            \frac{\mubc}{3 \nunu} \UP(\ups) - \frac{1}{2} \mubc^{2} \rr^{2}
        \Bigr\} 
        >
        \zzQ(\xx,\entrlb)
    \biggr)
    \le 
    \exp\bigl( - \xx \bigr) .
\label{PUPxxrr}
\end{EQA}
Let now a finite or separable set \( \Mubc \) and a function \( \pen(\mubc) \ge 1 \) 
be fixed such that 
\begin{EQA}[c]
%    \Mubcb 
%    \eqdef 
    \sum_{\mubc \in \Mubc} \ex^{- \pen(\mubc)} \le 2.
\label{Mubcbsup}
\end{EQA}    
One possible choice of the set \( \Mubc \) and the function \( \pen(\mubc) \) is to take
a geometric sequence \( \mubc_{k} = \mubc_{0} 2^{-k} \) with any fixed 
\( \mubc_{0} \) and define \( \pen(\mubc_{k}) = k = - \log_{2}(\mubc_{k}/\mubc_{0}) \) 
for \( k \ge 0 \).

Putting together the bounds \eqref{PUPxxrr} for different \( \mubc \in \Mubc \) 
yields the following result.

\begin{theorem}
\label{TPUPxxrr}
\label{Tglobalexpboundi}
\label{TLDglobsmoothLD}
Suppose \( (\CS\rr) \) and \eqref{Mubcbsup}.
Then for any \( \xx \ge 2 \), 
there exists a random set \( A(\xx) \) of a total probability at least 
\( 1 - 2 \ex^{- \xx} \), such that it holds on \( A(\xx) \)
for any \( \rr \) 
\begin{EQA}[c]
%    \sup_{\rr} 
    \sup_{\ups \in \B_{\rr}(\upss)}
    \sup_{\mubc \in \Mubc(\rr,\xx)}
    \Bigl[ 
        \frac{\mubc}{3 \nunu} \UP(\ups) - \frac{1}{2} \mubc^{2} \rr^{2}
        - 
        \bigl\{ 1 + \sqrt{\xx + \entrlb + \pen(\mubc)} \bigr\}^{2} 
    \Bigr]
    < 
    0,
\label{globsupqqQxxmu}
\end{EQA}
where 
\begin{EQA}[c]
    \Mubc(\rr,\xx) 
    \eqdef 
    \bigl\{ 
        \mubc \in \Mubc:  \,
        1 + \sqrt{\xx + \entrlb + \pen(\mubc)} \le \nunu \gm(\rr)/\mubc 
    \bigr\}.
\label{Mubcrr}
\end{EQA} 
\end{theorem}
Let \( \Ldrift(\ups) \) be a deterministic \emph{boundary} function.
%satisfying \( \Ldrift(\upss) = 0 \).
We aim at bounding the probability that a process \( \UP(\ups) \) hits 
this boundary on the set \( \Ups \). 
This precisely means the probability that 
\( \sup_{\ups \in \Ups} \bigl\{ \UP(\ups) - \Ldrift(\ups) \bigr\} \ge 0 \).
An important observation here is that multiplication by any positive factor 
\( \mubc \) does not change the relation. 
This allows to apply the multiscale result from Theorem~\ref{TPUPxxrr}.
For any fixed \( \xx \) and any \( \ups \in \B_{\rr}(\upss) \), define 
\begin{EQA}[c]
    \Lmgfb(\ups)
    \eqdef
    \sup_{\mubc \in \Mubc(\rr,\xx)} 
    \Bigl\{ 
        \frac{1}{3 \nunu} \mubc \Ldrift(\ups) 
        - \frac{1}{2} \mubc^{2} \rr^{2} - 2 \pen(\mubc) 
    \Bigr\} .
\label{Lmgfbups}
\end{EQA}    
%Below we suppose that \( \xx + \entrlb \ge 2.5 \) yielding

\begin{theorem}
\label{Thitting}
Suppose \( (\CS\rr) \), \eqref{Mubcbsup}, and \( \xx + \entrlb \ge 2.5 \).
Let, given \( \xx \), it hold 
\begin{EQA}[c]
    \Lmgfb(\ups)
    \ge 
    2 (\xx + \entrlb) ,
    \qquad 
    \ups \in \Ups .
\label{Lmgfbupen}
\end{EQA}    
Then 
\begin{EQA}[c]
    \P\Bigl( 
        \sup_{\ups \in \Ups} 
            \bigl\{ \UP(\ups)  - \Ldrift(\ups) \bigr\} 
        \ge  
        0
    \Bigr)
    \le 
    2 \ex^{-\xx} .
\label{hitprobxx}
\end{EQA}
\end{theorem}

\end{document}